\documentclass[11pt,oneside,leqno]{amsart}

\usepackage{amsxtra}
\usepackage{amsopn}
\usepackage{amsmath,amsthm,amssymb}
\usepackage{color}
\usepackage{amscd}
\usepackage{pifont}
\usepackage{amsfonts}
\usepackage{latexsym}
\usepackage{verbatim}
\usepackage{pb-diagram}

\theoremstyle{plain}
\newtheorem{theorem}{Theorem}[section]
\newtheorem*{theorem*}{Theorem}

\newtheorem{prop}[theorem]{Proposition}
\newtheorem{cor}[theorem]{Corollary}
\newtheorem{rem}[theorem]{Remark}
\newtheorem{ex}[theorem]{Example}
\newtheorem*{mt*}{Main Theorem}
\newcommand\R{{\mathbb R}}

\newcommand\pa[1]{\partial_{#1}}

\begin{document}
\title{Totally Geodesic and Parallel Hypersurfaces of G\"odel-type spacetimes}
\author{Giovanni Calvaruso}
\address{Giovanni Calvaruso: Dipartimento di Matematica e Fisica \lq\lq E. De Giorgi\rq\rq \\
Universit\`a del Salento\\
Prov. Lecce-Arnesano \\
73100 Lecce\\ Italy.}
\email{giovanni.calvaruso@unisalento.it}
\author{Lorenzo Pellegrino}
\address{Lorenzo Pellegrino: Dipartimento di Matematica e Fisica \lq\lq E. De Giorgi\rq\rq \\
Universit\`a del Salento\\
Prov. Lecce-Arnesano \\
73100 Lecce\\ Italy.}
\email{lorenzo.pellegrino@unisalento.it}
\author{Joeri Van der Veken}
\address{Joeri Van der Veken: Department of Mathematics \\
University of Leuven\\
Celestijnenlaan 200B \\
3001 Leuven \\  Belgium.}
\email{joeri.vanderveken@kuleuven.be}

\thanks{J. Van der Veken is supported by the Research Foundation–Flanders (FWO) and the National Natural Science Foundation of China (NSFC) under collaboration project G0F2319N, by the KU Leuven Research Fund under project 3E210539 and by the Research Foundation–Flanders (FWO) and the Fonds de la Recherche Scientifique (FNRS) under EOS Projects G0H4518N and G0I2222N}

\begin{abstract}
We classify parallel and totally geodesic hypersurfaces of the relevant class of G\"odel-type spacetimes, with particular regard to the homogeneous examples.
 \end{abstract}

\maketitle

\section{Introduction}

A submanifold $M$ of a pseudo-Riemannian manifold %$(\bar M, \bar g)$ 
is said to be \textit{parallel} if its second fundamental form $h$ (and hence, all the extrinsic invariants derived from it) is covariantly constant. Parallel submanifolds extend in a natural way the notion of \textit{totally geodesic} submanifolds, for which $h=0$ and so, the geodesics of a totally geodesic submanifold are also geodesics of the ambient space. Thus, the study of parallel and totally geodesic hypersurfaces of a given pseudo-Riemannian manifold is a natural problem, which enriches our knowledge and understanding of the geometry of the manifold itself.

Parallel hypersurfaces of a locally symmetric ambient space are locally symmetric, but this property does not exted to  more general ambient spaces. This fact makes it  particularly interesting to investigate parallel hypersurfaces of homogeneous spaces which are not locally symmetric. Moreover, in pseudo-Riemannian settings, hypersurfaces of different signatures can occur, so that their investigation is at the same time harder and richer than in the Riemannian case. 

Parallel surfaces have been intensively studied in three-dimensional Lorentzian ambient spaces. We may refer to 
\cite{CV1}-\cite{CV4} for several examples. Understandably, the study of parallel hypersurfaces becomes more difficult for  ambient spaces of higher dimension. On the other hand, also because of their relevance in Mathematical Physics, four-dimensional Lorentzian manifolds are natural candidates for this kind of study. Classifications of parallel hypersurfaces in some classes of four-dimensional Lorentzian and pseudo-Riemannian manifolds may be found in \cite{CV5},\cite{CSV},\cite{ChV},\cite{DV}. 
We may observe that in some cases, homogeneous pseudo-Riemannian four-manifolds  do not allow parallel and totally geodesic hypersurfaces (see for example \cite{CSV},\cite{DV}).

\smallskip
In this paper we classify parallel and totally geodesic hypersurfaces within the wide class of G\"odel-type spacetimes, with particular regard to the homogeneous examples. In 1949, G\"odel \cite{God} obtained a solution to Einstein field equations with cosmological constant for incoherent matter with rotation. In this interesting model,  closed timelike geodesics occur. 
G\"odel's example has been generalized to a family of rotating cosmological spacetimes, depending on two functions of one variable.%, known as  \emph{G\"odel-type spacetimes}, whose geometry depends on the influence of the rotation $\omega$.

In the usual notation of Theoretical Physics, which considers Lorentzian metrics of signature $(+,-,-,-)$, {\em G\"odel-type spacetimes} are described by the Lorentzian metrics 
\begin{equation}\label{GTmet}
g = [dt + H(r) d\phi]^2 - dr^2 - D^2(r) d\phi ^2 - dz^2,
\end{equation}
where $t$ is the time variable and $(r,\phi,z)$ are cylindrical coordinates, so that $r\geq 0$, $\phi \in \mathbb R$ 
(undetermined for $r=0$) and $z \in \mathbb R$. %Observe that the coefficients of $g$ only depend on the variable $r$. Moreover, 
As $\det (g)= -D^2(r)$, we have $D(r) \neq 0$ for the metric $g$ to be nondegenerate.

\subsection{G\"odel-type homogeneous spacetimes}
A G\"odel-type spacetime is homogeneous (i.e., it admits a group of isometries acting transitively on it) if and only if there exist some real constants $\alpha$ and $\omega$, such that 
\begin{equation}\label{GThom}
D''  = \alpha D,  \qquad  H'  = -2\omega D
\end{equation}
(see \cite{RG},\cite{RT}). %We may remark that Equation~\eqref{GTmet} defines a metric tensor (where $D(r) \neq 0$) for any functions $D,H$, required to be at least of class $\mathcal{C}^2$ to calculate their curvature. However, because of \eqref{GThom}, the homogeneous examples are determined by functions of class $\mathcal{C}^{\infty}$. G\"odel-type homogeneous spacetimes have been studied in different contexts. 
We may refer to the following works and references therein for some of the several topics which have been investigated for G\"odel-type homogeneous spacetimes: causality properties \cite{CRT},\cite{RA}, curvature collineations \cite{MNPV}, matter collineations  \cite{CS}, energy and momentum \cite{Sh}, geodesics \cite{GGKS}, the lightcone \cite{Dau}, geodesic connectedness \cite{BCF}, characterizations \cite{PS},  Klein-Gordon equations \cite{Ja}, Ricci solitons \cite{Ca}.

According to the sign of the real constants $\alpha$ and $\omega$, G\"odel-type homogeneous spacetimes are classified into the {following} four non-isometric classes \cite{CS}:

\medskip
{\em Class I:} $\alpha=m^2 >0$, $\omega \neq 0$. Then, the solution to~\eqref{GThom} is given by
\begin{equation}\label{GTI}
H(r)=\frac{2\omega}{m^2}[1-\cosh(mr)], \qquad D(r)=\frac{1}{m}[\sinh(mr)].
\end{equation}

\medskip
{\em Class II:} $\alpha=0$, $\omega \neq 0$. In this case,
\begin{equation}\label{GTII}
H(r)=-\omega r^2 , \qquad D(r)=r .
\end{equation}

\medskip
{\em Class III:} $\alpha=-\mu^2 <0$, $\omega \neq 0$. Then, 
\begin{equation}\label{GTIII}
H(r)=\frac{2\omega}{\mu^2}[\cos(\mu r)-1], \qquad D(r)=\frac{1}{\mu}[\sin(\mu r )].
\end{equation}

\medskip
{\em Class IV:} $\alpha\neq 0$, $\omega =0$. After a suitable change of coordinates, $H(r)=0$ and  $D(r)$ is either as in~\eqref{GTI} or as in~\eqref{GTIII}, depending on whether $\alpha=m^2 >0$ or $\alpha=-\mu^2<0$,  respectively.

%\begin{rem}
%{\em 

The standard G\"odel spacetime belongs to {\em Class~I} and is obtained for $m^2=2\omega^2$. 
Metrics belonging to {\em Class~II} and to {\em Class~I} with $0 < m^2 < 4 \omega^2$ have only one noncausal region. For 
$m^2 \geq 4 \omega ^2$ there exist no closed timelike curves. The limiting case $m^2=4\omega ^2$ is a completely causal homogeneous G\"odel-type spacetime. Homogeneous G\"odel-type metrics admit a five-dimensional group of isometries, except for \begin{itemize}
\item the limiting case $m^2=4\omega ^2$, which admits a seven-dimensional group of isometries \cite{RT}, and
\item  {\em Class~IV}, where the group of isometries is six-dimensional.
\end{itemize}
Indeed, the two cases listed above are locally symmetric (see also \cite{Ca}). As such, they are somewhat trivial for the actual study, as their parallel hypersurfaces are locally symmetric, too.

Metrics  in {\em Class~III} \ admit infinitely many causal and noncausal regions. Metrics in {\em Class~IV} are also known as {\em degenerate} G\"odel-type spacetimes, since their rotation is $\omega=0$. %The results of this paper gives another justification for considering them as degenerate. 
Finally, one excludes from the above classification  the trivial case $\alpha=\omega=0$, which corresponds to the Minkowski spacetime. Observe that parallel hypersurfaces in the Minkowski spacetime were classified in \cite{ChV}. {For this reason, throughout the paper we shall always exclude this case.}
%}\end{rem}

The paper is organized in the following way. In Section~2 we report the needed preliminary information about parallel and totally geodesic hypersurfaces and the description of the Levi-Civita connection and curvature of G\"odel-type spacetimes. 
In Section~3 we investigate the class of their hypersurfaces admitting a Codazzi second fundamental form, which include parallel (and totally geodesic) hypersurfaces. Finally, in Section~4 we deal with the classification and explicit description  of parallel and totally geodesic hypersurfaces of G\"odel-type spacetimes. We shall also point out several examples of proper constant mean curvature and minimal parallel {hypersurfaces}.

\section{Preliminaries}
\setcounter{equation}{0}

\subsection{Parallel and totally geodesic hypersurfaces}

Let $F : M^n \to \bar M^{n+1}$ be an isometric immersion of pseudo-Riemannian manifolds. We shall denote both metrics by $g$. Let $\xi$ be a unit normal vector field on the hypersurface, with $g(\xi,\xi) = \varepsilon \in \{ -1,1\}$. Denote by $\nabla ^M$ and $\nabla$ the Levi-Civita connections of $M^n$ and $\bar M^{n+1}$ respectively. Let $X$ and $Y$ be
vector fields on $M^n$ (we will always identify vector fields
tangent to $M^n$ with their images under $dF$). %The component of the vector field $\nabla_X Y$ tangent to $M^n$ is $\nabla ^M _X Y$ and we haveù
The well known {\em formula of Gauss}
\begin{align} \label{fG}
\nabla_X Y = \nabla ^M _X Y + h(X,Y) \xi
\end{align}
defines the \emph{second fundamental form} $h$ of the immersion, which is a
symmetric $(0,2)$-tensor field on $M^n$.

$M$ is said to be \emph{totally geodesic} if $h=0$. This is equivalent to the geometric property that every geodesic of $M$ is also a geodesic of the ambient space $\bar M$. %The prime examples are Euclidean subspaces of a Euclidean space.

%We now state the definition of a parallel hypersurface. 
Next, consider the covariant derivative $\nabla^M h$ of the second fundamental form, given by
$$
(\nabla^M h)(X,Y,Z) = X(h(Y,Z)) - h(\nabla^M _XY,Z) - h(Y,\nabla ^M _XZ)
$$
for all vector fields $X,Y,Z$ tangent to $M$. The hypersurface is said to be \emph{parallel}, or to have parallel
second fundamental form, if %and only if its cubic form vanishes identically, i.e.,
\begin{align} \label{nablahis0}
\nabla^M h = 0.
\end{align}
Clearly, totally geodesic hypersurfaces are parallel.

It is well known that $M$ is locally symmetric when $\nabla^M R=0$. Thus, condition \eqref{nablahis0} can be seen as the extrinsic analogue of local symmetry. Indeed, just like the Riemann-Christoffel curvature tensor $R$ contains all
the information on the intrinsic geometry of a pseudo-Riemannian manifold, the second fundamental form $h$ contains all the extrinsic geometric information, concerning how $M^n$ is immersed in $\bar M^{n+1}$. 

Such a correspondence is even more evident when the ambient space $\bar M^{n+1}$ is symmetric. In fact, let  $M^n \to \bar M^{n+1}$ denote a complete, connected and embedded hypersurface of a simply connected symmetric space. Then, $M^n$ is parallel if and only if for every $p \in M^n$ there exists an isometry $\sigma_p$ of $\bar M^{n+1}$ such that $\sigma_p(p) = p$, $d\sigma_p |_{T_pM^n} =  -\mathrm{id}_{T_pM^n}$, $d\sigma_p |_{T_p^{\perp}M^n} =
\mathrm{id}_{T_p^{\perp}M^n}$ and $\sigma_p M^n = M^n$ (see \cite{N}). Hence, $M^n$ itself is symmetric and the geodesic
symmetries of $M^n$ extend to isometries of the ambient space. More in general, parallel hypersurfaces of locally symmetric ambient spaces are again locally symmetric.

We now state the equations of Gauss and Codazzi, which follow from \eqref{fG} by a
straightforward computation. Denote by $R^M$ and $R$ the Riemann-Christoffel curvature tensors of $M^n$ and $\bar M^{n+1}$ respectively. The equations of Gauss and Codazzi then respectively read
\begin{align}
\label{eG}
& g(R (X,Y)Z,W) = g(R^M (X,Y)Z,W) + \varepsilon \left( h(X,Z)h(Y,W) - h(X,W)h(Y,Z) \right),\\
\label{eC} & g(R (X,Y)Z,\xi) = \varepsilon \left(
(\nabla^M h)(X,Y,Z) - (\nabla^M h)(Y,X,Z) \right),
\end{align}
where $X$, $Y$, $Z$ and $W$ are tangent to $M^n$. Throughout this paper, we will always use the sign convention 
$R(X,Y)=[\nabla_X, \nabla_Y]-\nabla_{[X,Y]}$.

The hypersurface is said to have a {\em Codazzi second fundamental form} if {$\nabla^{M} h$} is symmetric in its three arguments. Clearly, by equation~\eqref{eC}, this is equivalent to requiring that $R(X,Y)\xi=0$ for all vector fields $X$ and $Y$ on {$M$}. It is evident that totally geodesic and parallel hypersurfaces have a Codazzi second fundamental form.

Finally, another necessary condition for totally geodesic and parallel hypersurfaces is semi-parallelism. We recall that 
$M$ is said to be {\em semi-parallel} if $R^M \cdot h = 0$, where, for all  vectors $X,Y,Z,W$ tangent to $M$,
$$
(R^M \cdot h) (X,Y,Z,W) = -h(R^M (X,Y)Z, W)-h(Z,R^M (X,Y)W).
$$

\subsection{Curvature and connection of G\"odel-type metrics}
Let $g$ denote an arbitrary G\"odel-type metric $g$, as described by  \eqref{GTmet} with respect to the coordinate system 
$(x_1,x_2,x_3,x_4) = (t , r, \phi, z)$. We shall denote by $\{ \pa {_i}=\frac{\partial}{\partial x_i}\}$ the basis of coordinate vector fields. Then, it is easy to check that 
\begin{equation}\label{frame}
E_1= \pa {1}, \quad E_2= \pa {2}, \quad E_3=- \frac{H}{D} \pa {1}+ \frac{1}{D} \pa {3}, \quad E_4= \pa {4}
\end{equation}
form a pseudo-orthonormal basis for $g$, namely $-g(E_1, E_1)=g(E_2, E_2)=g(E_3, E_3)=g(E_4, E_4)=-1$ and $g(E_i,E_j)=0$ if $i\neq j$.
In addition, we can observe that the only possibly  non-vanishing Lie bracket of the above vector fields is given by
\begin{equation}\label{bra}
[E_2,E_3]=-\frac{H'}{D}E_1-\frac{D'}{D}E_3.
\end{equation}
We can now describe the Levi-Civita connection $\nabla$ of $g$ with respect to the basis $\{E_i\}$. By means of the 
{\em Koszul formula} and \eqref{bra} we get
\begin{equation}\label{nabla}
\begin{array}{llll}
\nabla_{E_1}E_1=0, &\;
\nabla_{E_2}E_1=-\frac{H'}{2D}E_3, &\;
\nabla_{E_3}E_1=\frac{H'}{2D}E_2,&\;
\nabla_{E_4}E_1=0, 
 \\[6pt]
\nabla_{E_1}E_2=-\frac{H'}{2D}E_3, &\;
\nabla_{E_2}E_2=0, &\;
\nabla_{E_3}E_2=\frac{H'}{2D}E_1+\frac{D'}{D}E_3, &\;
\nabla_{E_4}E_2=0, 
 \\[6pt]
  \nabla_{E_1}E_3=\frac{H'}{2D}E_2, &\;
\nabla_{E_2}E_3=-\frac{H'}{2D}E_1, 
&\;
\nabla_{E_3}E_3=-\frac{D'}{D}E_2, 
&\;
\nabla_{E_4}E_3=0,
 \\[6pt]
  \nabla_{E_1}E_4=0, &\;
\nabla_{E_2}E_4=0, 
&\;
\nabla_{E_3}E_4=0, 
&\;
\nabla_{E_4}E_4=0.
\end{array}
\end{equation}
We now consider the curvature tensor of $g$. Starting from \eqref{nabla} we find by a direct calculation that with respect to 
$\{E_i \}$, the curvature tensor  is completely determined by  the following possibly non-vanishing components
\begin{equation}\label{R}
\begin{array}{ll}
R(E_1,E_2)E_1=\left( \frac{H'}{2D}\right)^2E_2, &\; 
R(E_1,E_2)E_2=\left( \frac{H'}{2D}\right)^2E_1- \left( \frac{H'}{2D}\right)' E_3,
\\[7pt]
R(E_1,E_2)E_3=\left( \frac{H'}{2D}\right)' E_2, &\; 
R(E_1,E_3)E_1=\left( \frac{H'}{2D}\right)^2 E_3,
\\[7pt]
R(E_1,E_3)E_3=\left( \frac{H'}{2D}\right)^2 E_1, &\;
R(E_2,E_3)E_1=-\left( \frac{H'}{2D}\right)' E_2,
\\[7pt]
R(E_2,E_3)E_2=-\left( \frac{H'}{2D}\right)' E_1+\frac{3H'^2-4DD''}{4D^2}E_3, &\;
R(E_2,E_3)E_3=-\frac{3H'^2-4DD''}{4D^2}E_2,
\end{array}
\end{equation}

\section{Hypersurfaces with a Codazzi second fundamental form}
\setcounter{equation}{0}

Let $F: M \rightarrow (\bar M,g)$ denote the immersion of a hypersurface into a G\"odel-type spacetime and $\xi$ the unit  normal vector field to the hypersurface. The following result gives some necessary algebraic conditions on the components of 
$\xi$ with respect to the frame $\{E_1, E_2, E_3, E_4\}$ on $\bar M$, in order for $M$ to have a Codazzi second fundamental form.

\begin{theorem}\label{t1}
Let $F: M \rightarrow \bar M$ be a hypersurface with a Codazzi second fundamental form and $\xi$ the unit normal vector field,  with $g(\xi, \xi)=\varepsilon \in \{-1, 1\}$. Consider the pseudo-orthonormal frame $\{E_i\}$ on $\bar M$ defined in 
\eqref{frame} and set 
$$
f_1=\left(\frac{H'}{2D}\right)^2,  \qquad f_2=-\left(\frac{H'}{2D}\right)', \qquad f_3=3f_1- \frac{D''}{D}.
$$
Then, every point of $M$ has a neighborhood $U\subseteq M$ on which one of the following conditions holds:
\begin{itemize}
\vspace{3pt}\item[(I)] $\xi=E_4$;
\vspace{3pt}\item[(II)] $\xi=E_2$;
\vspace{3pt}\item[(III)] $\xi=\cos \theta E_2+\sin \theta E_3$ for some function $\theta:U\rightarrow \R$ and $f_2=0$;
\vspace{3pt}\item[(IV)] $\xi=aE_1+cE_3$ for some functions $a,c:U\rightarrow \R$ and $ac(f_1+f_3)-(a^2+c^2)f_2=0$;
\vspace{3pt}\item[(V)] $\xi=aE_1+dE_4$ for some functions $a,d:U\rightarrow \R$ and $f_1=0$;
\vspace{3pt}\item[(VI)] $\xi=aE_1+bE_2+cE_3$ for some functions $a,b,c:U \rightarrow \R$ and $f_2=f_1+f_3=0$.
\end{itemize}
\end{theorem}

\begin{rem}
Note that in the statement of the Theorem {\em\ref{t1}}, $f_j$ means $f_j\circ F_{|_U}$. In order to simplify the presentation, we use this notation from now on in the rest of  he paper. The same notation is used for functions $H$ and $D$.
\end{rem}

\begin{proof}
We first observe that functions $f_1, f_2, f_3$ are not independent from each other. In fact, if $f_1=0$ then $f_2=0$. On the other hand, if $f_2=0$, then $f_1$  is a constant. Moreover, if $f_1=f_2=f_3=0$, then we get $\alpha=\omega=0$ (see 
\eqref{GThom}), that is, the trivial case of a Minkowski spacetime, which we will exclude.

Consider now $\xi= aE_1+bE_2+cE_3+dE_4$, for some functions $a, b, c, d: U\rightarrow \R$ satisfying 
$a^2-b^2-c^2-d^2 =g(\xi,\xi)=\varepsilon =\pm 1 \neq 0$. Then, the following vector fields are tangent to the hypersurface:
$$
\begin{array}{ll}
X_1=bE_1+aE_2, &\quad X_4=cE_2-bE_3,\\[6pt]
X_2=cE_1+aE_3, &\quad X_5=dE_2-bE_4,\\[6pt]
X_3=dE_1+aE_4, &\quad X_6=dE_3-cE_4.
\end{array}
$$
If $h$ is Codazzi, equation \eqref{eC} yields that $R(X_i,X_j)\xi=0$ for every $i,j\in\{1,\ldots,6\}$. In particular,
\begin{align}
\label{123}
0&=R(X_1,X_2)\xi=a^2bf_2E_1-a(ac(f_1+f_3)-(a^2+c^2)f_2)E_2+ab(af_1-cf_2+af_3)E_3,
\\
\label{10}
0&=R(X_1,X_5)\xi=b^2df_1E_1+bd(af_1-cf_2)E_2+b^2df_2E_3,
\\
\label{16}
0&=R(X_2,X_3)\xi=acdf_1E_1-a^2df_1E_3,
\\
\label{19}
0&=R(X_2,X_4)\xi=-abcf_2E_1+c(ac(f_1+f_3)-(a^2+c^2)f_2)E_2-bc(af_1-cf_2+af_3)E_3,
\\
\label{35}
0&=R(X_4,X_5)\xi=b^2df_2E_1+bd(af_2-cf_3)E_2+b^2df_3E_3
\end{align}
We will treat separately two cases, depending on whether $a=0$ or $a\neq0$.

\medskip
\textbf{Case 1: {${a=0}$}.} In this case, equation \eqref{19} implies that $c^3 f_2=0$. Hence, we have the following two subcases.

\textit{Case 1.1: $a=c=0$.} By equation \eqref{10} we then have $b^2df_1=0$. If $b=0$ we get case (I) in the statement; for $d=0$ we recover the case (II). Finally, if $f_1=0$, then $f_2=0$ and from equation \eqref{35} we also find $f_3=0$, so that we get the case we excluded.

\textit{Case 1.2: $a=f_2=0$.} It follows from equation \eqref{10} that $b^2df_1=0$. If $d=0$, then $g(\xi, \xi)=-b^2-c^2=-1$ and we obtain the case (III) in the statement. If either $f_1=0$ or $b=0$ we recover the previous cases.

\textbf{Case 2: ${a\neq0}$}. In this case, it follows from equation \eqref{123} that $bf_2=0$. So, we distinguish two subcases.

\textit{Case 2.1: $b=0$.} Equation \eqref{16}, as $a \neq 0$, yields that $df_1=0$. If $d=0$, taking into account equation \eqref{19} we find case (IV) in the statement. If $f_1=0$, then from equation \eqref{123} we get $a^2cf_3=0$ which, excluding the case of the Minkowski spacetime,  yields case (V) in the statement.

\textit{Case 2.2: $f_2=0$.} In this case, it follows from  equation \eqref{10} that $b^2df_1=0$. If $d=0$, from equation  
\eqref{123} we obtain the case (VI) in the statement. If either $f_1=0$ or $b=0$ we recover some of the cases we already obtained.
\end{proof}

In the following results we provide an explicit description for the immersion $F: M \to \bar M$ of hypersurfaces corresponding to  types (I),(II),(III) and (V) listed in Theorem~\ref{t1}.

\begin{theorem}\label{FtcfI}
Let $F: M \rightarrow \bar M$ denote a hypersurface of type {\em (I)} in
Theorem~{\em\ref{t1}}. Then, %up to isometries of the ambient space, 
the immersion can be described explicitly in local coordinates as
$$
F(u_1, u_2, u_3) = (u_1, u_2, u_3, 0).
$$
In particular, these timelike hypersurfaces are totally geodesic.
\end{theorem}
\begin{proof}
As  $\xi=E_4$, taking into account \eqref{frame}, we get that the tangent space to $M$ at every point is given by 
span$\{E_1,E_2,E_3\}=$span$\{\pa 1, \pa 2, \pa 3 \}$. With respect to coordinates $(u_1,u_2,u_3)=(x_1,x_2,x_3)$ on $M$, after applying a translation in the $x_4$-direction, we obtain the required parametrization for $F$.

Finally, it  follows from \eqref{nabla} that {the component of $\nabla _{E_i} E_j$ along $\xi={E_4}$ vanishes for all $i,j=1,2,3$}. In particular, by the formula of Gauss \eqref{fG} this implies that $h=0$. Thus, $M$ is totally geodesic.
\end{proof}

\begin{theorem}\label{FtcfII}
Let $F: M \rightarrow \bar M$ denote a hypersurface of type {\em (II)} in
Theorem~{\em\ref{t1}}. Then, %up to isometries of the ambient space, 
the immersion can be described explicitly in local coordinates as
{$$
F(u_1, u_2, u_3) = \left( u_1-\dfrac{H}{D}u_2, c, \dfrac{1}{D}u_2, u_3 \right),
$$ where $c$ is a real constant}. In particular, these timelike hypersurfaces are parallel and flat.
\end{theorem}
\begin{proof}
Since $\xi=E_2$, vector fields $E_1,E_3,E_4$ span the tangent space to $M$ at every point.

A direct calculation, using {\eqref{frame}} and \eqref{nabla}, gives
\begin{equation}\label{nablaII}
\begin{array}{lll}
\nabla_{E_1}E_1=0, &\quad
\nabla_{E_3}E_1=\frac{H'}{2D}\xi,& \quad
\nabla_{E_4}E_1=0, 
 \\[6pt]
  \nabla_{E_1}E_3=\frac{H'}{2D}\xi, &\quad
\nabla_{E_3}E_3=-\frac{D'}{D}\xi, 
&\quad
\nabla_{E_4}E_3=0,
 \\[6pt]
  \nabla_{E_1}E_4=0, &\quad
\nabla_{E_3}E_4=0, 
&\quad
\nabla_{E_4}E_4=0.
\end{array}
\end{equation}
Since vector fields in \eqref{nablaII} are normal to $M$, by the Gauss formula \eqref{fG} we get
\begin{equation}\label{nablaMII}
\nabla^M_{E_i}E_j=0, \quad i,j \in\{1, 3, 4\}
\end{equation}
and so, $\nabla^M =0$. In particular, $M$ is flat and the vector
fields $E_1  = \pa {u_1}$, $E_3  = \pa {u_2}$ and $E_4  = \pa {u_3}$ may be taken as coordinate vector fields on $M$.

Denote now by $F: M \rightarrow \bar M, \; (u_1, u_2, u_3)\mapsto(F_1(u_1, u_2, u_3), \ldots, F_4(u_1, u_2, u_3))$ the immersion of the hypersurface in the local coordinates introduced above. By \eqref{frameIII} and \eqref{frame}, we obtain
\begin{equation} \label{eqpaII}
\begin{array}{ll}
(\pa {u_1} F_1, \pa {u_1} F_2, \pa {u_1} F_3, \pa {u_1} F_4)=& (1,0,0,0),\\[6pt]
(\pa {u_2} F_1, \pa {u_2} F_2, \pa {u_2} F_3, \pa {u_2} F_4)=& {(-\frac{H}{D},0,\frac{1}{D},0)},\\[6pt]
(\pa {u_3} F_1, \pa {u_3} F_2, \pa {u_3} F_3, \pa {u_3} F_4)=& (0,0,0,1).
\end{array}
\end{equation}
Integrating \eqref{eqpaII} we find
{$$%\begin{equation} \label{solpaII}
%\begin{array}{l}
 F_1= u_1-\frac{H}{D}u_2+c_1,\quad F_2= c_2,\quad F_3= \frac{1}{D}u_2+c_3,\quad F_4= u_3+ c_4
%\end{array}
$$}%\end{equation}}
for some real constants $c_1$, $c_2$, $c_3$ and $c_4$. After a reparametrization, we obtain the immersion given in the statement.

Again from \eqref{nablaII} and the Gauss formula, we get that the second fundamental form is determined by
\begin{equation}\label{hII}
\begin{array}{lll}
h(E_1,E_1)=0, &\qquad h(E_1,E_3)=\frac{H'}{2D}, &\qquad h(E_3,E_4)=0,
\\[6pt]
h(E_1,E_4)=0, &\qquad h(E_3,E_3)=- \frac{D'}{D}, &\qquad h(E_4,E_4)=0.
\end{array}
\end{equation}
Observe that, by \eqref{hII}, $h$ depends only on $x_2$. Therefore, by \eqref{nablaMII} and \eqref{hII}, we get at once that  
$\nabla^M h=0$, that is, $M$ is parallel.
\end{proof}

\begin{rem}
{\em From \eqref{hII} we can observe that hypersurfaces of type (II) described in Theorem~\ref{FtcfII} are totally geodesic if and only if $D'(r)=H'(r)=0$, that is, $D$ and $H$ are constants, so that $\bar M$ is isometric to the Minkowski space.
}\end{rem}

\begin{prop}\label{FtcfIII}
Let  $F: M \rightarrow \bar M$ denote a hypersurface of type {\em (III)} listed in Theorem~{\em\ref{t1}}. Then there exist local coordinates $(u_1, u_2, u_3)$ on $M$ such that %, up to isometries of the ambient space,
the immersion is explicitly given by
\begin{align*}
F(u_1, u_2, u_3) = (u_1+G_1(u_2), G_2(u_2), G_3(u_2), u_3)
\end{align*}
for some functions $G_1$, $G_2$, $G_3$ satisfying $D^2\left(G'_1+G'_2\right)^2= (H-1)^2 \left(1-(G'_3)^2\right)$. In particular, these timelike hypersurfaces are flat.
\end{prop}

\begin{proof}
Since $\xi=\cos \theta E_2+\sin \theta E_3$ for some function $\theta:U\rightarrow \R$, vector fields
\begin{equation}\label{frameIII}
Y_1 = E_1,\quad Y_2 = \sin \theta E_2-\cos \theta E_3, \quad Y_3 = E_4
\end{equation}
span the tangent space to $M$ at each point.  Using \eqref{frameIII} and \eqref{nabla}, a direct calculation gives
\begin{equation}\label{nablaIII}
\begin{array}{lll}
\nabla_{Y_1}Y_1=0, &\quad
\nabla_{Y_2}Y_1=-\frac{H'}{2D}\xi, &\quad
\nabla_{Y_3}Y_1=0, 
 \\[6pt]
\nabla_{Y_1}Y_2=\left( -\frac{H'}{2D}+ Y_1(\theta)\right)\xi, &\quad
\nabla_{Y_2}Y_2=\left( Y_2(\theta)-\frac{D'}{D}\cos \theta \right)\xi, &\quad
\nabla_{Y_3}Y_2=0, 
 \\[6pt]
  \nabla_{Y_1}Y_3=0, &\quad
\nabla_{Y_2}Y_3=0, 
&\quad
\nabla_{Y_3}Y_3=0.
\end{array}
\end{equation}
Since all these vector fields are normal to $M$, using the Gauss formula, we get
\begin{equation}\label{nablaMIII}
\nabla^M_{Y_i}Y_j=0, \quad i,j \in\{1, 2, 3\}.
\end{equation}
Thus, $M$ is flat and vector fields $Y_i = \pa {u_i}$, $i=1,2,3$, are coordinate vector fields on $M$.

Next, from \eqref{frameIII}, \eqref{nablaIII} and the Gauss formula, we get that the second fundamental form is determined by
\begin{equation}\label{hIII}
\begin{array}{lll}
h(Y_1,Y_1)=0, &\quad h(Y_1,Y_3)=0, &\quad h(Y_2,Y_3)=0,
\\[6pt]
h(Y_1,Y_2)=-\frac{H'}{2D}, &\quad h(Y_2,Y_2)=Y_2(\theta)-\frac{D'}{D}\cos \theta , &\quad h(Y_3,Y_3)=0,
\end{array}
\end{equation}
where we used the symmetry of $h$, which is equivalent to $\pa {u_1} \theta=0$ with respect to the coordinates introduced above. Moreover, $f_1$ is a constant, as $f_2=0$.

Denote now by $F: M \rightarrow \bar M, \; (u_1, u_2, u_3)\mapsto(F_1(u_1, u_2, u_3), \ldots, F_4(u_1, u_2, u_3))$ the immersion of the hypersurface in the local coordinates introduced above. By using \eqref{frameIII} and \eqref{frame}, we obtain
\begin{equation} \label{eqpaIII}
\begin{array}{ll}
(\pa {u_1} F_1, \pa {u_1} F_2, \pa {u_1} F_3, \pa {u_1} F_4)=& (1,0,0,0),\\[6pt]
(\pa {u_2} F_1, \pa {u_2} F_2, \pa {u_2} F_3, \pa {u_2} F_4)=& {(\frac{H}{D}\cos \theta,\sin \theta,-\frac{1}{D}\cos \theta,0)},\\[6pt]
(\pa {u_3} F_1, \pa {u_3} F_2, \pa {u_3} F_3, \pa {u_3} F_4)=& (0,0,0,1).
\end{array}
\end{equation}
Observe that, by  \eqref{eqpaIII}, $\partial_{u_3} \theta =0$. In fact, $0=\partial_{u_3} (\partial_{u_2} F_2)
=(\partial_{u_3} \theta) \cos \theta$. Integrating \eqref{eqpaIII}, we find
\begin{equation} \label{solpaIII}
{\begin{array}{ll}
\displaystyle
F_1= u_1+ \int_{c_1}^{u_2} \frac{H}{D} \cos \theta \,ds,\quad
F_2= \int_{c_2}^{u_2}\sin \theta\, ds,\quad
F_3= -\int_{c_3}^{u_2}\frac{\cos \theta}{D}\, ds,\quad
F_4= u_3+ c_4,
\end{array}}
\end{equation}
for some real constants $c_i, i=1,2,3,4$. After a reparametrization, we obtain the immersion given in the statement.
\end{proof}

\begin{prop}\label{FtcfV}
Let $F: M \rightarrow \bar M$ denote a hypersurface of type {\em (V)} listed in Theorem~{\em\ref{t1}}. Then there exist local coordinates $(u_1, u_2, u_3)$ on $M$ such that, up to isometries of the ambient space, the immersion is explicitly given by
\begin{align*}
F(u_1, u_2, u_3) = 
 \left(G_1(u_1)-H u_3, u_2,u_3,G_4(u_1)\right)
 \end{align*}
for some functions $G_1$, $G_4$ satisfying  $(G'_1)^2-(G'_4)^2=-\varepsilon$.
\end{prop}

\begin{proof}
Since $\xi=a E_1+d E_4$ for some function {$a, d:U\rightarrow \R$} and $||\xi||^2=a^2-d^2= \varepsilon =\pm 1$, there exists some smooth function $\theta: U \mapsto \R$ such that 
$$
\xi=\frac{e^{\theta}+\varepsilon e^{-\theta}}{2} E_1+ \frac{e^{\theta}-\varepsilon e^{-\theta}}{2} E_4.
$$
Then, the following vector fields span the tangent space to $M$ at each point:
\begin{equation}\label{frameV}
Y_1 = \frac{e^{\theta}-\varepsilon e^{-\theta}}{2} E_1+ \frac{e^{\theta}+\varepsilon e^{-\theta}}{2} E_4 ,\qquad Y_2 = E_2, \qquad Y_3 = E_3.
\end{equation}
Using \eqref{frameV}, the condition $f_1=0$ and \eqref{nabla}, a direct calculation gives
\begin{equation}\label{nablaV}
\begin{array}{lll}
\nabla_{Y_1}Y_1=Y_1(\theta)\xi, &\quad
\nabla_{Y_2}Y_1=Y_2(\theta)\xi, &\quad
\nabla_{Y_3}Y_1=Y_3(\theta)\xi, 
 \\[6pt]
\nabla_{Y_1}Y_2=0, &\quad
\nabla_{Y_2}Y_2=0, &\quad
\nabla_{Y_3}Y_2=\frac{D'}{D} Y_3, 
 \\[6pt]
  \nabla_{Y_1}Y_3=0, &\quad
\nabla_{Y_2}Y_3=0, 
&\quad
\nabla_{Y_3}Y_3=- \frac{D'}{D} Y_2,
\end{array}
\end{equation}
From \eqref{nablaV}, using the Gauss formula \eqref{fG}, we get that the Levi-Civita connection on $M$ is completely determined by the following possibly non-vanishing components:
\begin{equation}\label{nablaMV}
\nabla^M_{Y_3}Y_2=\frac{D'}{D} Y_3, \qquad \nabla^M_{Y_3}Y_3=- \frac{D'}{D} Y_2.
\end{equation}
Next, from \eqref{frameV}, \eqref{nablaV} and the Gauss formula, we conclude that the second fundamental form is determined by
\begin{equation}\label{hV}
h(Y_1,Y_1)=Y_1(\theta), \quad h(Y_i,Y_j)=0, \; \text{for all} \; (i,j)\neq(1,1),
\end{equation}
where we took into account the symmetry condition for $h$, which yields $Y_2(\theta)=Y_3(\theta)=0$. Moreover, as $f_2=0$, we deduce that $f_1$ is a constant.

We now look for a system of local coordinates $(u_1,u_2, u_3)$ on $M$, such that
\begin{equation}\label{cooMV}
%\begin{cases}
    \pa {u_1}=Y_1,\qquad 
    \pa {u_2}=Y_2, \qquad
    \pa {u_3}=\alpha Y_2+ \beta Y_3
%\end{cases}
\end{equation}
for some smooth functions $\alpha, \beta$ on  $M$. Requiring that $[ \pa {u_2},\pa {u_3}]=0$, we get
\begin{align*}
\begin{cases}
    Y_2(\alpha)=0 ,\\[5pt]
   Y_2(\beta)=\frac{D'}{D}\beta .
\end{cases}
\end{align*}
Observe that we only need one solution for $\alpha$ and $\beta$ in the system above in order to find a coordinate system 
$(u_1, u_2, u_3)$ on the surface $M$. So, we take $\alpha=0$ and $\beta=D(r)=D(u_2)$.

With respect to the coordinates introduced above, the symmetry conditions for $h$ read $\pa {u_2} \theta=\pa {u_3} \theta=0$. Therefore,  $ \theta=\theta(u_1)$.

Denote now by $F: M \rightarrow \bar M, \; (u_1, u_2, u_3)\mapsto(F_1(u_1, u_2, u_3), \ldots, F_4(u_1, u_2, u_3))$ the immersion of the hypersurface in the local coordinates introduced above. By using \eqref{frameV} and \eqref{frame}, we obtain
\begin{equation} \label{eqpaV}
\begin{array}{ll}
(\pa {u_1} F_1, \pa {u_1} F_2, \pa {u_1} F_3, \pa {u_1} F_4)=& (\frac{e^{\theta}-\varepsilon e^{-\theta}}{2},0,0,\frac{e^{\theta}+\varepsilon e^{-\theta}}{2}),\\[6pt]
(\pa {u_2} F_1, \pa {u_2} F_2, \pa {u_2} F_3, \pa {u_2} F_4)=& (0,1,0,0),\\[6pt]
(\pa {u_3} F_1, \pa {u_3} F_2, \pa {u_3} F_3, \pa {u_3} F_4)=& {(-H,0,1,0)}.
\end{array}
\end{equation}
Integrating \eqref{eqpaV} we have
$$%\begin{equation} \label{solpaV}
\begin{array}{ll}
F_1= \int^{u_1}_{c_1} \frac{e^{\theta}-\varepsilon e^{-\theta}}{2}\, ds-H u_3,\quad &
F_2= u_2+c_2,\\[5pt]
F_3= u_3+c_3,\quad &
F_4= \int^{u_1}_{c_4} \frac{e^{\theta}+\varepsilon e^{-\theta}}{2}\, ds,
\end{array}
$$%\end{equation}
for some real constants $c_i, i=1,2,3,4$. After a reparametrization, we obtain the immersion given in the statement.
\end{proof}

\section{Parallel and totally geodesic hypersurfaces}
\setcounter{equation}{0}

We shall now proceed with the classification of parallel and totally geodesic surfaces in the different cases listed in 
Theorem~\ref{t1}. Cases (I) and (II) have already been completely treated in the previous section.  

\bigskip
{
\textbf{Case (III): $\xi=\cos \theta E_2+\sin \theta E_3$ for some function $\theta:U\rightarrow \R$ and $f_2=0$.}

Recall that the description obtained in Proposition~\ref{FtcfIII} applies to these hypersurfaces. 
There, we have constructed Euclidean coordinate vector fields
$$%\begin{equation}
%\begin{cases}
    \pa {u_1}=Y_1,\qquad    \pa {u_2}=Y_2, \qquad     \pa {u_3}= Y_3
%\end{cases}
$$
and gave explicit expressions for the second fundamental form in \eqref{hIII}. Starting from \eqref{nablaMIII} and \eqref{hIII}, it is easily seen that the immersion is parallel if and only if
\begin{equation}\label{parIII cond1}
Y_2\left(Y_2(\theta) - \cos \theta \frac{D'}{D}\right)=0,
\end{equation}
since $\theta$ does not depend on {either} $x_1$ or $x_4$. Then, using that $\pa{u_2}=\sin \theta \pa 2- \cos \theta \pa 3$, we can rewrite \eqref{parIII cond1} as follows:
\begin{equation}\label{parIII cond2}
\pa{u_2} (\ln (D \cos \theta))=-\lambda \tan \theta
\end{equation}
for $\lambda \in \R$. Therefore, the parametrization of $M$ given in the Proposition~\ref{FtcfIII}, together with the condition \eqref{parIII cond2}, characterize completely parallel hypersurfaces of $\bar M$ in the case~(III).

In addition we now describe the totally geodesic examples. It follows at once from \eqref{hIII} that $M$ is totally geodesic if and only if
$$
Y_2(\theta) - \cos \theta \frac{D'}{D}=0,
$$
which by integration gives explicitly
\begin{equation}\label{tgcondIII}
\theta= \arccos\left(\frac{\rho}{D}\right),
\end{equation}
where $\rho$ is a real constant. By \eqref{tgcondIII}, equation \eqref{eqpaIII} for the immersion $F: M \rightarrow \bar M, \; (u_1, u_2, u_3)\mapsto(F_1(u_1, u_2, u_3), \ldots, F_4(u_1, u_2, u_3))$  of the hypersurface in the local coordinates introduced above, now reads 
\begin{equation} \label{solpaIIItg}
{\begin{array}{ll}
\displaystyle
F_1= u_1+ \rho\int_{c_1}^{u_2} \frac{  H}{D^2} \,ds, &\quad
F_2= \int_{c_2}^{u_2}\sqrt{1-\left(\frac{\rho}{D}\right)^2}\, ds,\quad \\[10pt]
\displaystyle
F_3= -\rho\int_{c_3}^{u_2}\frac{1}{D^2}\, ds, &\quad
F_4= u_3+ c_4
\end{array}}
\end{equation}
for some real constants $c_i, i=1,2,3,4$.
}

{\begin{ex}
Consider the special case where $\theta$ is constant on $M$. Then, condition \eqref{parIII cond1} becomes
$$
\frac{\pa {u_2} (D)}{D}=-\lambda \tan \theta,
$$
with $\lambda \in \R$, which is satisfied for 
$$
D=\rho e^{-\lambda \tan \theta u_2}
$$ 
{for some} real constant $\rho$. Taking into account $f_2=0$, this example corresponds to a homogeneous metric in the \textit{class I}, with $\alpha=\lambda^2$. Moreover, from $H'=-2\omega D$, with $\omega$ a real constant, {we obtain}
$$
H=\frac{2\omega \rho}{\lambda \cos \theta} e^{-\lambda \tan \theta\, u_2} +k
$$
with $k$ a real constant. Then, {up to translation, the immersion of the hypersurface in the local coordinates introduced above is given by  $F: M \rightarrow \bar M, \; (u_1, u_2, u_3)\mapsto(F_1(u_1, u_2, u_3), \ldots, F_4(u_1, u_2, u_3))$,
with}
$$
\begin{array}{ll}
F_1= u_1+ \frac{2\omega}{\lambda} u_2+ k \frac{\cos^2 \theta}{\lambda \rho \sin \theta} e^{\lambda \tan \theta u_2}, \quad
&F_2= \sin \theta \, u_2,  \\ [5pt]
F_3= -\frac{\cos^2 \theta}{\lambda \rho \sin \theta} e^{\lambda \tan \theta u_2}, \quad
&F_4= u_3.
\end{array}
$$
\end{ex}
}

{\smallskip
\textbf{Case (IV): $\xi=aE_1+cE_3$ with $a,c:U\rightarrow \R$ satisfying $ac(f_1+f_3)-(a^2+c^2)f_2=0$.}

\smallskip
Since $\xi=a E_1+c E_3$ for some function $a, c:U\rightarrow \R$ and $||\xi||^2=a^2-c^2= \varepsilon =\pm 1$, we distinguish two subcases.

\smallskip
\textit{Case (IV.i): $\varepsilon=-1$.}

In this case there exists some smooth function $\theta: U \mapsto \R$ such that 
$$
\xi=\sinh \theta E_1+\cosh \theta E_3.
$$
Then, the following vector fields span the tangent space to $M$ at each point:
\begin{equation}\label{frameIV}
Y_1 = E_2, \qquad Y_2 = E_4, \qquad Y_3 = \cosh \theta E_1+\sinh \theta E_3.
\end{equation}
The condition  $ac(f_1+f_3)-(a^2+c^2)f_2=0$ now becomes
$$
\tanh(2\theta)= \frac{2 f_2}{f_1+f_3},
$$
that means $\theta=\theta(x_2)$ since $f_1$, $f_2$ and $f_3$ {only depend} on $x_2$. 

{\begin{rem}
From the definition of $f_1,f_2,f_3$ we easily see that in this case, $H' \neq D'$. In fact, if $H'=D'$, Then, the above equation yields $\tanh(2\theta)=1$, which {cannot} occur.
\end{rem}
Using $\theta=\theta(x_2)$ in} \eqref{frameIV} and \eqref{nabla}, a direct calculation gives
\begin{equation}\label{nablaIV}
\begin{array}{lll}
\nabla_{Y_1}Y_1=0, &\quad
\nabla_{Y_2}Y_1=0, &\quad
\nabla_{Y_3}Y_1=A(x_2)Y_3+B(x_2)\xi, 
 \\[6pt]
\nabla_{Y_1}Y_2=0, &\quad
\nabla_{Y_2}Y_2=0, &\quad
\nabla_{Y_3}Y_2=0, 
 \\[6pt]
  \nabla_{Y_1}Y_3=(Y_1(\theta)-\frac{H'}{2D})\xi, &\quad
\nabla_{Y_2}Y_3=0, 
&\quad
\nabla_{Y_3}Y_3=A(x_2)Y_1,
\end{array}
\end{equation}
where $$
A(x_2)=\frac{\sinh \theta}{D}\left(H'\cosh \theta-D'\sinh \theta\right), \quad B(x_2)=\frac{D'}{D}\sinh \theta \cosh \theta-\frac{H'}{2D} (\sinh^2 \theta+\cosh^2 \theta).
$$
From \eqref{nablaIV}, using the Gauss formula \eqref{fG}, we get that the Levi-Civita connection on $M$ is completely determined by the following possibly non-vanishing components:
\begin{equation}\label{nablaMIV}
\nabla^M_{Y_3}Y_1=A(x_2)Y_3, \qquad \nabla^M_{Y_3}Y_3=A(x_2) Y_1.
\end{equation}
Next, from \eqref{frameIV}, \eqref{nablaIV} and the Gauss formula, we conclude that the second fundamental form is determined by
\begin{equation}\label{hIV}
h(Y_1,Y_3)=h(Y_3,Y_1)=Y_1(\theta)-\frac{H'}{2D}, \quad h(Y_i,Y_j)=0, \; \text{for all} \; (i,j)\neq(1,3),
\end{equation}
where, by the symmetry condition for $h$, $Y_1(\theta)-\frac{H'}{2D}=B(x_2)$. Moreover, since $\theta$ depends only on $x_2$, we get
$$
\theta'=\frac{\sinh \theta}{D}\left(D'\cosh \theta-H'\sinh \theta\right).
$$
In order to find the cases where $M$ is parallel, we first impose $M$ to be semi-parallel by requiring $R^M\cdot h=0$, which yields:
$$%\begin{equation}\label{semiparallelIVi}
\left(A'+A^2\right)h(Y_1,Y_3)=0.
$$%\end{equation}
Then, we have that  either $A'+A^2=0$ or $h(Y_1,Y_3)=0$. 

\smallskip
\textit{First case: $A'+A^2=0$.}

In this case, by \eqref{nablaMIV} it easily follows that $R^M=0$ and so, $M$ is flat. Moreover, integrating $A'+A^2=0$  we get explicitly
$$ A(x_2)=\frac{1}{x_2}+k,
$$
where $k$ is a real constant.
From \eqref{hIV}, it is now straightforward  that $\nabla^M h = 0$ if and only if $\theta''=(\frac{H'}{2D})'$, whence 
$\theta'=\frac{H'}{2D}+\lambda$ for some real constant  $\lambda$. Thus, since $H' \neq D'$ we have 
$$
\theta=%\begin{cases}
\frac{1}{2}\ln\frac{2D\lambda\pm \sqrt{4D^2\lambda^2+(D')^2-(H')^2}}{D'-H'}. %\\[5pt]
%\text{constant} \quad &\mathrm{if}\, D'= H'.
%\end{cases}
$$
We now look for a system of local coordinates $(u_1,u_2, u_3)$ on $M$, such that
\begin{equation}\label{cooMIV}
%\begin{cases}
    \pa {u_1}=Y_1,\qquad 
    \pa {u_2}=Y_2, \qquad
    \pa {u_3}=\alpha Y_2+ \beta Y_3
%\end{cases}
\end{equation}
for some smooth functions $\alpha, \beta$ on  $M$. Requiring that $[ \pa {u_2},\pa {u_3}]=0$, we get
\begin{align*}
\begin{cases}
    Y_1(\alpha)=0 ,\\[3pt]
   Y_1(\beta)=A \beta .
\end{cases}
\end{align*}
{A solution for $\alpha$ and $\beta$ in the system above is given by} $\alpha=0$ and $\beta=x_2 e^{k x_2}$.

With respect to the coordinates  $(u_1,u_2,u_3)$ we just introduced, conditions for $\theta$ read $\pa {u_2} \theta=\pa {u_3} \theta=0$. Therefore,  $ \theta=\theta(u_1)$.

Denote now by $F: M \rightarrow \bar M, \; (u_1, u_2, u_3)\mapsto(F_1(u_1, u_2, u_3), \ldots, F_4(u_1, u_2, u_3))$ the immersion of the hypersurface in the local coordinates introduced above. By using \eqref{frameIV} and \eqref{frame}, we obtain
\begin{equation} \label{eqpaIV}
\begin{array}{ll}
(\pa {u_1} F_1, \pa {u_1} F_2, \pa {u_1} F_3, \pa {u_1} F_4)=& (0,1,0,0),\\[6pt]
(\pa {u_2} F_1, \pa {u_2} F_2, \pa {u_2} F_3, \pa {u_2} F_4)=& (0,0,0,1),\\[6pt]
(\pa {u_3} F_1, \pa {u_3} F_2, \pa {u_3} F_3, \pa {u_3} F_4)=&\beta_{|_F} (\cosh \theta- \frac{H}{D}\sinh \theta,0,\frac{1}{D}\sinh \theta,0).
\end{array}
\end{equation}
Integrating \eqref{eqpaIV} we obtain
$$ \label{solpaIV}
\begin{array}{ll}
F_1= \left(\cosh \theta- \frac{H}{D}\sinh \theta\right) (u_1+c_2) e^{k (u_1+c_2)}u_3+c_1,\quad &
F_2= u_1+c_2,\\[5pt]
F_3= \frac{ \sinh \theta}{D} (u_1+c_2) e^{k (u_1+c_2)}u_3+c_3,\quad &
F_4= u_2+c_4,
\end{array}
$$
for some real constants $c_i, i=1,2,3,4$.

\smallskip
\textit{Second case: $h(Y_1,Y_3)=0$.}

In this case, from \eqref{hIV} we deduce that $M$ is totally geodesic. Moreover, as $H' \neq D'$, from $Y_1(\theta)-\frac{H'}{2D}=B(x_2)=0$  we get
$$
\theta=%\begin{cases}
\frac{1}{4}\ln\left(\frac{D'+H'}{D'-H'} \right). %$\mathrm{if}\, D'\neq H', %\\[5pt]
%\frac{1}{2}\ln(D)+ C\quad &\mathrm{if}\, D'= H',
%\end{cases}
$$
%where $C$ is a real constant. 
%Observe that if $D'=H'$, from $B(x_2)=0$ we conclude that $D'=H'=0$, so that $\bar M$ is isometric to the Minkowski spacetime and we exclude this case.
%
We then choose a system of local coordinates $(u_1,u_2, u_3)$ on $M$, such that
\begin{equation}\label{cooMIVb}
%\begin{cases}
    \pa {u_1}=Y_1,\qquad 
    \pa {u_2}=Y_2, \qquad
    \pa {u_3}=\alpha Y_2+ \beta Y_3,
%\end{cases}
\end{equation}
for some smooth functions $\alpha, \beta$ on  $M$. Requiring that $[ \pa {u_2},\pa {u_3}]=0$, we get
\begin{align*}
\begin{cases}
    Y_1(\alpha)=0 ,\\[3pt]
   Y_1(\beta)=A \beta .
\end{cases}
\end{align*}
{We choose as solution for $\alpha$ and $\beta$ in the system above}  $\alpha=0$ and {$\beta=\exp{\int A(x_2) dx_2}$, where we set
$$
A(x_2)=\frac{D'-\sqrt{(D')^2-(H')^2}}{2D}.
$$
}
With respect to the coordinates  $(u_1,u_2,u_3)$  introduced above, the conditions for $\theta$ read $\pa {u_2} \theta=\pa {u_3} \theta=0$. Therefore,  $ \theta=\theta(u_1)$.

Denote now by $F: M \rightarrow \bar M, \; (u_1, u_2, u_3)\mapsto(F_1(u_1, u_2, u_3), \ldots, F_4(u_1, u_2, u_3))$ the immersion of the hypersurface in the local coordinates introduced above. By using \eqref{frameIV} and \eqref{frame}, we obtain
\begin{equation} \label{eqpaIVb}
\begin{array}{ll}
(\pa {u_1} F_1, \pa {u_1} F_2, \pa {u_1} F_3, \pa {u_1} F_4)=& (0,1,0,0),\\[6pt]
(\pa {u_2} F_1, \pa {u_2} F_2, \pa {u_2} F_3, \pa {u_2} F_4)=& (0,0,0,1),\\[6pt]
(\pa {u_3} F_1, \pa {u_3} F_2, \pa {u_3} F_3, \pa {u_3} F_4)=&\beta_{|_F} (\cosh \theta- \frac{H}{D}\sinh \theta,0,\frac{1}{D}\sinh \theta,0).
\end{array}
\end{equation}
Integrating \eqref{eqpaIVb} we obtain
$$%\begin{equation} \label{solpaIVb}
\begin{array}{ll}
F_1= \beta \left(\cosh \theta- \frac{H}{D}\sinh \theta\right)u_3+c_1,\quad &
F_2= u_1+c_2,\\[5pt]
F_3= \beta \frac{\sinh \theta}{D}u_3+c_3,\quad &
F_4= u_2+c_4
\end{array}
$$%\end{equation}
for some real constants $c_i, i=1,2,3,4$.
}
\smallskip

\textit{Case (IV.ii): $\varepsilon=1$.}

This case is completely analogous to the case \textit{(IV.ii)}, taking now
$$
\xi=\cosh \theta E_1+\sinh \theta E_3.
$$
Then, in this case, the following vector fields span the tangent space to $M$ at each point:
\begin{equation}\label{frameIVb}
Y_1 = E_2, \qquad Y_2 = E_4, \qquad Y_3 = \sinh \theta E_1+\cosh \theta E_3.
\end{equation}
Again we recover $\theta=\theta(x_2)$. The connection on $M$ and the second fundamental form have analogous description as in \eqref{nablaMIV} and \eqref{hIV}, taking into account the differences expressed in \eqref{frameIVb} and that we now have
$$
A(x_2)=\frac{\cosh \theta}{D}\left(D'\cosh \theta-H'\sinh \theta\right), \quad B(x_2)=\frac{H'}{2D} (\sinh^2 \theta+\cosh^2 \theta)-\frac{D'}{D}\sinh \theta \cosh \theta.
$$
Next, this case the condition of semi-parallelism gives two different subcases: either $Y_1(A)+A^2=0$ and $M$ flat, or $M$ is totally geodesic. The description of parallel and totally geodesic hypersurfaces is obtained by the corresponding ones for the previous case, simply interchanging $\sinh$ with $\cosh$ in the parametrization.

\smallskip
\textbf{Case (V): $\xi=aE_1+dE_4$ for some functions $a,d:U\rightarrow \R$ and $f_1=0$.}
\smallskip

These hypersurfaces have been described in general in Proposition~\ref{FtcfV}.
There, we constructed Euclidean coordinate vector fields
$$%\begin{equation}
%\begin{cases}
    \pa {u_1}=Y_1,\qquad    \pa {u_2}=Y_2, \qquad     \pa {u_3}=D Y_3
%\end{cases}
$$
and gave explicit expressions for the second fundamental form in \eqref{hV}. Starting from \eqref{nablaMV} and \eqref{hV}, it is easily seen that the immersion is parallel if and only if $\pa {u_1} \theta$ is constant. As $\theta$ does not depend on $u_2,u_3$, we then have $\theta(u_1) = k_1u_1 + k_2$, for some real constants $k_1,k_2$. We treat separately the cases $k_1\neq 0$ and $k_1=0$.

\newpage
{\textit{Case (V.i): $k_1\neq 0$.}}

In this case, denote by $F: M \rightarrow \bar M: \; (u_1, u_2, u_3)\mapsto(F_1(u_1, u_2, u_3), \ldots, F_4(u_1, u_2, u_3))$ the immersion of the hypersurface in the local coordinates introduced above. By using \eqref{frameV} and \eqref{frame}, we obtain
\begin{equation} \label{eqpaVpar}
\begin{array}{ll}
(\pa {u_1} F_1, \pa {u_1} F_2, \pa {u_1} F_3, \pa {u_1} F_4)=& (\frac{e^{\theta}-\varepsilon e^{-\theta}}{2},0,0,\frac{e^{\theta}+\varepsilon e^{-\theta}}{2}),\\[6pt]
(\pa {u_2} F_1, \pa {u_2} F_2, \pa {u_2} F_3, \pa {u_2} F_4)=& (0,1,0,0),\\[6pt]
(\pa {u_3} F_1, \pa {u_3} F_2, \pa {u_3} F_3, \pa {u_3} F_4)=& {(-H,0,1,0)}.
\end{array}
\end{equation}
Integrating \eqref{eqpaVpar} we find
\begin{equation} \label{solpaVpar}
{\begin{array}{ll}
F_1= \frac{e^{k_1 u_1+k_2}+\varepsilon e^{-k_1 u_1-k_2}}{2k_1}-H u_3+ c_1,\quad &
F_2= u_2+c_2,\\[6pt]
F_3= u_3+c_3,\quad &
F_4= \frac{e^{k_1 u_1+k_2}-\varepsilon e^{-k_1 u_1-k_2}}{2k_1}+ c_4
\end{array}}
\end{equation}
for some real constants $c_1$, $c_2$, $c_3$ and $c_4$. After a reparametrization, we obtain the immersion explicitly given by

{\begin{align*}
F(u_1, u_2, u_3) = 
\begin{cases}
 \left(c \cosh (u_1)+H u_3, u_2,-u_3, c\sinh(u_1)\right) \; \mathrm{for}\, M\, \text{spacelike},\\[6pt]
 \left(c\sinh (u_1)+H u_3, u_2,-u_3, c\cosh(u_1)\right) \; \mathrm{for}\, M\, \text{timelike},
\end{cases}
\end{align*}
for some real constant {$c$}.}

\smallskip
{\textit{Case (V.ii): $k_1= 0$.}

In this case $\theta=k_2$ is a constant. Denote again by $F: M \rightarrow \bar M: \; (u_1, u_2, u_3)\mapsto(F_1(u_1, u_2, u_3), \ldots, F_4(u_1, u_2, u_3))$ the immersion of the hypersurface in the local coordinates introduced above. By using \eqref{frameV} and \eqref{frame}, we obtain
$$%\begin{equation} \label{eqpaVparthk}
\begin{array}{ll}
(\pa {u_1} F_1, \pa {u_1} F_2, \pa {u_1} F_3, \pa {u_1} F_4)=& (\frac{e^{k_2}-\varepsilon e^{-k_2}}{2},0,0,\frac{e^{k_2}+\varepsilon e^{-k_2}}{2}),\\[6pt]
(\pa {u_2} F_1, \pa {u_2} F_2, \pa {u_2} F_3, \pa {u_2} F_4)=& (0,1,0,0),\\[6pt]
(\pa {u_3} F_1, \pa {u_3} F_2, \pa {u_3} F_3, \pa {u_3} F_4)=& {(-H,0,1,0)},
\end{array}
$$%\end{equation}

which, by integration, yield
\begin{equation} \label{solpaVparthk}
{\begin{array}{ll}
F_1= \frac{e^{k_2}-\varepsilon e^{-k_2}}{2}u_1-H u_3+ c_1,\quad &
F_2= u_2+c_2,\\[6pt]
F_3= u_3+c_3,\quad &
F_4= \frac{e^{k_2}+\varepsilon e^{-k_2}}{2}u_1+ c_1,
\end{array}}
\end{equation}
for some real constants $c_i$, $i=1,2,3,4$.

After a reparametrization we obtain
{\begin{align*}
F(u_1, u_2, u_3) = 
\begin{cases}
 \left(\sinh (k_2)u_1-H u_3, u_2,u_3,\cosh(k_2)u_1\right) \quad \mathrm{for}\, M\, \text{spacelike}, \\[6pt]
 \left(\cosh(k_2)u_1-H u_3, u_2,u_3, \sinh (k_2)u_1\right) \quad \mathrm{for}\, M\, \text{timelike}.
\end{cases}
\end{align*}}

\medskip
\textbf{Case (VI): $\xi=aE_1+bE_2+cE_3$ for some functions $a,b,c:U \rightarrow \R$ and $f_2=f_1+f_3=0$.}
\smallskip

We first observe that conditions $f_2=f_1+f_3=0$ imply that
$$
H'=-2\omega D,\quad D''=4 \omega^2 D,
$$
where $\omega$ is a real constant. This corresponds to 
\begin{itemize}
\item  the limiting case within the homogeneous case {\em I)} if $\omega\neq 0$;%, which is the most symmetric homogeneous G\"odel metric, as it admits a seven-dimensional group of isometries. 
\item the Minkowski spacetime if $\omega = 0$.
\end{itemize}
In the limiting case, $\bar M$ decomposes as the product $N^3_1(c)\times \R$ of a Lorentzian three-manifold of constant sectional curvature $c>0$ and a real line \cite{Ca}.

Since $\xi$ is a unit vector field tangent $N^3_1(c)$, applying a suitable isometry, it suffices to consider the cases 
$\xi=E_1$ (timelike) and $\xi=E_2$ (spacelike).

\smallskip
\textit{Case (VI.i): $\xi=E_1$.}

Since $\xi=E_1$, vector fields $E_2,E_3,E_4$ span the tangent space to $M$ at every point.

A direct calculation, using \eqref{frame} and \eqref{nabla}, gives
\begin{equation}\label{nablaVI}
\begin{array}{lll}
\nabla_{E_2}E_2=0, &\;
\nabla_{E_3}E_2=\frac{H'}{2D}\xi+\frac{D'}{D}E_3, &\;
\nabla_{E_4}E_2=0, 
 \\[6pt]
\nabla_{E_2}E_3=-\frac{H'}{2D}\xi, 
&\;
\nabla_{E_3}E_3=-\frac{D'}{D}E_2, 
&\;
\nabla_{E_4}E_3=0,
 \\[6pt]
\nabla_{E_2}E_4=0, 
&\;
\nabla_{E_3}E_4=0, 
&\;
\nabla_{E_4}E_4=0.
\end{array}
\end{equation}
From \eqref{nablaVI} and the Gauss formula, the symmetry of the second fundamental form implies that $H'=0$ and so, 
$h=0$, that is, $M$ is totally geodesic. However, in this case $\bar M$ is isometric to the Minkowski space. Therefore, we shall exclude this case.
\smallskip

\textit{Case (VI.ii): $\xi=E_2$.}

This is a special case of case (II). We already know {from} Proposition~\ref{FtcfII} that $M$ is parallel and flat. Requiring that $M$ is totally geodesic, from \eqref{hII} we deduce again $f_1=f_2=f_3=0$, so that $\bar M$ is isometric to the Minkowski space. 

\medskip
{The above calculations and conclusions are summarized in the following main classification results of totally geodesic and parallel hypersurfaces of  G\"odel-type spacetimes.

\begin{theorem}
Let $F: M \rightarrow \bar M$ be a totally geodesic
hypersurface of a G\"odel-type spacetime. Consider the coordinates
$(x_1, x_2, x_3, x_4)$ on $\bar M$ introduced in Section {\em 2}. Then there exist local coordinates $(u_1, u_2, u_3)$ on $M$, such that up to isometries, the immersion is given by one of the following expressions.

\medskip
$\bullet$ For any value of $H$ and $D$:
\begin{itemize}
\item[(a)] $F(u_1, u_2, u_3) = (u_1, u_2, u_3, 0)$ and $M$ is timelike.
\end{itemize}

\medskip
$\bullet$ If $\frac{H'}{2D}$ is constant:
\begin{itemize}
\item[(b)] $F(u_1, u_2, u_3)= \left( u_1+ \rho\int_{0}^{u_2} \frac{H}{D^2} \,ds, \int_{0}^{u_2}\sqrt{1-\left(\frac{\rho}{D}\right)^2}\, ds, -\rho\int_{0}^{u_2}\frac{1}{D^2}\, ds, u_3 \right),$

\smallskip\noindent
for a real constant $\rho$ and $M$ is timelike.
\end{itemize}

$\bullet$ If $\left|\frac{\left(\frac{H'}{D}\right)'}{\left(\frac{H'}{D}\right)^2- \frac{D''}{D}}\right| <1$ and 
$(D')^2 >(H')^2$:
\begin{itemize}
\item[(c)] {
$F(u_1, u_2, u_3) = 
\begin{cases}
 \left( \left(\tanh \theta- D\right)u_3, u_1, u_3, u_2 \right) \; 
{\text{and}\, M\, \text{is spacelike},}\\[6pt]
 \left( \left(\coth \theta- D\right)u_3, u_1, u_3, u_2 \right) \; 
{\text{and}\, M\, \text{is spacelike}}
\end{cases}$

\medskip\noindent
with $\theta=\dfrac{1}{4}\ln\left(\frac{D'+H'}{D'-H'} \right)$.}
\end{itemize}

$\bullet$ If $H$ is constant:
\begin{itemize}{
\item[(d)] $F(u_1, u_2, u_3) = 
\begin{cases}
 \left(\tanh \theta\, u_1-H u_3, u_2,u_3, u_1\right) \quad \text{and}\, M\, \text{is spacelike}, \\[6pt]
 \left(\coth\theta\, u_1-H u_3, u_2,u_3,u_1\right) \quad \text{and}\, M\, \text{is timelike}
\end{cases}$\\[6pt]
for some real constant $\theta$}.
\end{itemize}
\end{theorem}

\begin{theorem}\label{mainpa}
Let $F: M \rightarrow \bar M$ be a {proper (i.e., not totally geodesic)} parallel hypersurface of a G\"odel-type spacetime. Consider the coordinates $(x_1, x_2, x_3, x_4)$ on $\bar M$ introduced in Section {\em 2}. Then there exist local coordinates $(u_1, u_2, u_3)$ on $M$, such that up to isometries the immersion is given by one of the following expressions:

\smallskip
$\bullet$ For any value of $H$ and $D$:
\begin{itemize}
\vspace{5pt}\item[(1)] {$F(u_1, u_2, u_3) = \left( u_1-\dfrac{H}{D}u_2, c, \dfrac{1}{D}u_2, u_3 \right),$\\[6pt]
where {$c$} is a real constant and $M$ is timelike.}
\end{itemize}

$\bullet$ If $\frac{H'}{2D}$ is constant:
\begin{itemize}
\item[(2)] $\displaystyle F(u_1, u_2, u_3) = \left( u_1+ \int_{0}^{u_2} \frac{H}{D} \cos \theta \,ds,  \int_{0}^{u_2}\sin \theta\, ds, -\int_{0}^{u_2}\frac{\cos \theta}{D}\, ds, u_3 \right),$\\[6pt]
with $\pa{u_2} (\ln (D \cos \theta))=-\lambda \tan \theta$ { for {some real constant $\lambda$} and $M$ is timelike}.
\end{itemize}

$\bullet$ If $\left|\frac{\left(\frac{H'}{D}\right)'}{\left(\frac{H'}{D}\right)^2- \frac{D''}{D}}\right| <1$ and $4D^2\lambda^2+(D')^2 \geq (H')^2$ for some real constant $\lambda$:
\begin{itemize}
\item[(3)]
$F(u_1, u_2, u_3) = 
\begin{cases}
\left( \left(\sinh \theta- \frac{H}{D}\cosh \theta\right) u_1 e^{k u_1}u_3, u_1, \frac{ \cosh \theta}{D} u_1 e^{k u_1}u_3, u_2\right) \; {\text{and}\, M\, \text{is spacelike},}\\[6pt]
\left( \left(\cosh \theta- \frac{H}{D}\sinh \theta\right) u_1 e^{k u_1}u_3, u_1, \frac{ \sinh \theta}{D} u_1 e^{k u_1}u_3, u_2\right) \; {\text{and}\, M\, \text{is timelike},}
\end{cases}$

\bigskip\noindent
where $k$ is a real constant and
$$
\theta=%\begin{cases}
\frac{1}{2}\ln\frac{2D\lambda\pm \sqrt{4D^2\lambda^2+(D')^2-(H')^2}}{D'-H'}; %\quad &\mathrm{if}\, D'\neq H', \\[5pt]
%\text{constant} \quad &\mathrm{if}\, D'= H',
%\end{cases}
$$
\end{itemize}

\medskip
$\bullet$ If $H$ is constant:
\begin{itemize}

\vspace{5pt}\item[(4)] {$F(u_1, u_2, u_3) = 
\begin{cases}
 \left(\cosh (u_1)-H u_3, u_2,u_3, \sinh(u_1)\right) \; {\text{and}\, M\, \text{is spacelike},}\\[6pt]
 \left(\sinh (u_1)-H u_3, u_2,u_3, \cosh(u_1)\right) \; {\text{and}\, M\, \text{is timelike}}.
\end{cases}$}
\end{itemize}
\end{theorem}

{{\em Minimal hypersurfaces} are the well-known {generalizations} of totally geodesic hypersurfaces, defined by the vanishing of the trace of $h$. More in general,  constant mean curvature (CMC) hypersurface are defined requiring that the trace of the second fundamental form is constant. There is an ever growing interest toward these classes of hypersurfaces. With regard to the class of parallel hypersurfaces of G\"odel-type spacetime we classified, a straightforward calculation leads to the following.}

\begin{cor}
Let $M$ be a parallel hypersurface of a G\"odel-type spacetime, as described in Theorem {\em\ref{mainpa}}. Then, $M$ is a minimal { but not totally geodesic} hypersurface if and only if:

\medskip
\begin{itemize}
\item[(a)]either $M$ corresponds to case {\em (1)} with $D$ constant, or

\smallskip
\item[(b)] $M$ corresponds to case {\em (3)}
\end{itemize}
\end{cor}

\begin{rem}
The above case {\em (b)} extends to the whole class of Codazzi hypersurfaces of type (IV) in Theorem {\em\ref{t1}}{, that is, all such hypersurfaces are minimal}.
\end{rem}

\begin{cor}
Let $M$ be a parallel hypersurface of a G\"odel-type spacetime, as described in Theorem {\em\ref{mainpa}}. Then, $M$ is a hypersurface of constant mean curvature (CMC{$\neq0$}) if and only if one of the following occurs:

\medskip
\begin{itemize}
\item[(a)] $M$ corresponds to case {\em (1)} with $D'$ constant;

\smallskip
\item[(b)] $M$ corresponds to case {\em (2)};

\smallskip
\item[(c)] $M$ corresponds to case {\em (4)}.
\end{itemize}
\end{cor}

}

\end{document}